\documentclass{article}

\newtheorem{theorem}{Theorem}

\newtheorem{corollary}[theorem]{Corollary}

\newtheorem{lemma}[theorem]{Lemma}

\newtheorem{question}[theorem]{Question}

\newenvironment{proof}[1][Proof]{\noindent\textbf{#1.} }{\ \rule{0.5em}{0.5em}}
\usepackage{amssymb}
\usepackage{amsmath}

\begin{document}

\title{The hyperspace of non-cut subcontinua of graphs}

\author{Alejandro Illanes, Ver\'onica Mart\'inez-de-la-Vega and Jorge E. Vega}
\maketitle

\abstract Given a continuum $X$, let $C(X)$ be the hyperspace of all subcontinua of $X$. We consider the hyperspace $NC^{*}(X)=\{A\in C(X):X\setminus A$ is connected$\}$. In this paper we prove that the only locally connected continua $X$ for which $NC^{*}(X)$ is compact are the arcs and the simple closed curves. We also characterize the finite graphs $G$ for which $NC^{*}(G)$ is connected. 
\bigskip



\section{Introduction}

A {\it continuum} is a compact connected metric space with more than one point. A {\it subcontinuum} of a continuum $X$ is a nonempty closed connected subset of $X$, so one-point sets in $X$ are subcontinua of $X$. Given a continuum $X$, we consider the following hyperspaces:

\begin{center}
$2^{X}=\{A\subset X:A$ is closed and nonempty$\}$,

$C(X)=\{A\in 2^{X}: A$ is connected$\}$,

$NC(X)=\{A\in C(X): X\setminus A$ is connected and $\operatorname{int}_{X}(A)=\emptyset\}$,  and

$NC^{*}(X)=\{A\in C(X): X\setminus A$ is connected$\}$.
\end{center}

These hyperspaces are endowed with the Hausdorff metric $H$ \cite[Definition 2.1 p.11]{in}, or equivalently with the Vietoris topology \cite[Definition 1.1 p.3]{in}.

The hyperspace $NC(X)$, called {\it the hyperspace of non-cut subcontinua}, was introduced in \cite{ees}. An {\it arc} is a continuum homeomorphic to the unit interval $[0,1]$. A {\it simple closed curve} is a continuum homeomorphic to the unit circle, centered at the origin, in the plane. A {\it finite graph} is a continuum that can be written as a finite union of arcs such that the intersection of every pair of them is finite.

The hyperspace $NC^{*}(X)$ was studied in \cite{hmmv} for the cases that $X$ is a finite graph or a dendrite. In \cite{hmmv}, among other results, it was shown:

\begin{itemize}

\item If $X$ is a dendrite (a locally connected continuum without simple closed curves) with its set of end-points dense, then $NC^{*}(X)$ is homeomorphic to the space of irrational numbers \cite[Theorem 5.7]{hmmv}.

\item If $G$ is a finite graph, then $NC^{*}(G)$ is locally connected \cite[Theorem 4.6]{hmmv}.

\item If $G$ is a finite graph, then $NC^{*}(G)$ is compact if and only if either $G$ is an arc or $G$ is a simple closed curve.
\end{itemize} 

In \cite{hmmv}, the authors also included the following question.

{\bf Question} \cite[Question 3.12]{hmmv}. Let $X$ be a locally connected continuum. If $X$ is homeomorphic to $NC^{*}(X)$, is $X$ an arc?

In this paper we prove (Theorem \ref{one}) that if $X$ is a locally connected continuum and $NC^{*}(X)$
is a closed subset of $C(X)$, then $X$ is either an arc or a simple closed
curve.
As a consequence we obtain a positive answer for the question above.

We also obtain a complete characterization of those finite graphs for which $NC^{*}(X)$ is connected.

\section{Peano continua}

\begin{theorem}\label{one} Let $X$ be a locally connected continuum. Suppose that $NC^{*}(X)$
is a closed subset of $C(X)$. Then $X$ is either an arc or a simple closed
curve.
\end{theorem}

\begin{proof}
Suppose to the contrary that $X$ is neither an arc nor a simple closed curve.
By Theorem 9.31 p.156 of \cite{n}, there exists two points $w,z\in X$ such that $w\neq
z$ and the open set $P=X\setminus \{w,z\}$ is connected. By Corollary 9.29 p.154 
of \cite{n}, there exists a point $y\in X\setminus \{w,z\}$ such that $\{y\}\in
NC^{*}(X)$. 

\smallskip

{\bf Claim 1.} There exist $E\in C(X)$ and open connected subsets $W,Z$ of
$X$ such that $w\in W$, $z\in Z$, $y\notin \operatorname{cl_{X}}(W\cup Z)$,
$\operatorname{cl_{X}}(W)\cap \operatorname{cl_{X}}(Z)=\emptyset$ and $X\setminus
E=W\cup Z$.

We prove Claim 1. Let $W_{0},Z_{0}$ be open subsets of $X$ such that $w\in
W_{0}$, $z\in Z_{0}$, $y\notin \operatorname{cl_{X}}(W_{0}\cup Z_{0})$ and
$\operatorname{cl_{X}}(W_{0})\cap \operatorname{cl_{X}}(Z_{0})=\emptyset$.

Let $E_{0}=X \setminus (W_{0}\cup Z_{0})$. For each $p\in E_{0}$, let $U_{p}$
be an open connected subset of $X$ such that $p\in U_{p}\subset \operatorname{cl_{X}}(U_{p})\subset
X\setminus \{w,z\}$. By the compactness of $E_{0}$, there exist $n\in\mathbb{N}$
and $p_{1},\ldots,p_{n}\in E_{0}$ such that $E_{0}\subset U_{p_{1}}\cup\cdots\cup
U_{p_{n}}$. Set $E_{1}= \operatorname{cl_{X}}(U_{p_{1}}\cup\cdots\cup U_{p_{n}})$.
Then $E_{1}$ is a closed subset of $X$, $E_{1}\cap \{w,z\}=\emptyset$, and
the number of components of $E_{1}$ is finite.

Since $E_{1}\subset P$ and $P$ is an open connected subset of $X$, the local
connectedness of $X$ implies that for every pair of indices $i,j\in \{1,\ldots,n\}$,
with $i\neq j$, there exists an arc $\alpha_{i,j}$ in $P$, connecting $p_{i}$
to $p_{j}$. Let 
\begin{center}
$E_{2}=E_{1}\cup\bigcup\{\alpha_{i,j}:i,j\in \{1,\ldots,n\}$ and $i\neq j\}$.
\end{center}   

Observe that $E_{2}$ is a subcontinuum of $X$ such that $E_{0}\subset E_{2}\subset
P$. Let $W$, $Z$ be the components of $X\setminus E_{2}$ such that $w\in
W$ and $z\in Z$. Then $W$ and $Z$ are open connected subsets of $X$ such
that $W\subset W_{0}$ and $Z\subset Z_{0}$. Let \begin{center}
$E=E_{2}\cup(\bigcup \{V:V$ is a component of $X\setminus E_{2}$ and $V\notin\{W,Z\}\})$.
\end{center}

Then $X\setminus E=W\cup Z$. Thus, $E$ is closed in $X$. Given a component
$V$ of $X\setminus E_{2}$, by \cite[Theorem 5.7, p.75]{n}, $\operatorname{cl_{X}}(V)\cap E_{2}\neq\emptyset$.
This implies that $E$ is a subcontinuum of $X$. This completes the proof
of Claim 1.

\bigskip

Let
\begin{center}
$\mathcal{A}=\{A\in NC^{*}(X):y\in A\subset E\}$.
\end{center}

{\bf Claim 2.} $\mathcal{A}$ contains maximal elements with respect to the inclusion.

In order to prove Claim 2, by the Brouwer's Reduction Theorem, we only need to show that if $\{A_{n}\}_{n=1}^{\infty}$ is a sequence of elements of $\mathcal{A}$ such
that $A_{1}\subset A_{2}\subset \cdots$, then there exists $A\in\mathcal{A}$ such that for each $n\in\mathbb{N}$, $A_{n}\subset A$. Set $A=\operatorname{cl_{X}(}\bigcup\{A_{n}:n\in\mathbb{N}\})=\operatorname{lim_{n\rightarrow
\infty}}A_{n}$. By hypothesis, $A\in NC^{*}(X)$, it follows that $A\in\mathcal{A}$.
Hence, we obtain that there exists a maximal element $A_{0}$ in $\mathcal{A}$.
Claim 2 is proved.
\bigskip

Since $X\setminus A_{0}$ is an open connected subset of $X$ and it contains
the points $w$ and $z$, there exists an arc $\alpha$ in $X\setminus A_{0}$
that joins $w$ to $z$. Since $X\setminus E=W\cup Z$ is disconnected, we have
that $A_{0}\subsetneq E$. Since $A_{0}\cap \alpha=\emptyset$, using an order
arc from $A_{0}$ to $E$, it is possible to construct a subcontinuum $B$ of
$E$ such that $A_{0}\subsetneq B\subset E\setminus \alpha$.
Since $X$ is locally connected, there exists an open connected subset $V$
of $X$ such that $\alpha\subset V\subset \operatorname{cl_{X}}(V)\subset
X\setminus B$. Let $Q=V\cup W\cup Z$. Then $Q$ is open and connected in $X$
and $Q\cap B=\emptyset$. Let $C$ be the component of $X\setminus Q$ that
contains $B$. By \cite[3.3 Ch. V, p.118]{d}, $C\in NC^{*}(X)$ and $A_{0}\subsetneq B \subset C\subset C\subset E$. This contradicts the maximality of $A_{0}$ and finishes the proof of the
theorem.
\end{proof}

As a consequence of Theorem \ref{one}, we obtain the following generalization of \cite[Corollary 3.10]{hmmv}, where this result was proved for finite graphs.

\begin{corollary}\label{two}
Let $X$ be a locally connected continuum. Then we have the following.

\noindent (1) $NC^{*}(X)$ is a continuum if and only if $X$ is either an arc or a simple closed curve.

\noindent (2) $NC^{*}(X)$ is homeomorphic to $X$ if and only if $X$ is an arc.

\noindent (3) $NC^{*}(X)=C(X)$ if and only if $X$ is a simple closed curve.
\end{corollary}

It is easy to show that if $X$ is a compactification of the ray $[0,\infty)$, then $NC^{*}(X)$ is compact; if $X$ is an indecomposable continuum, then $NC^{*}(X)=C(X)$; and if $X$ the Warsaw circle, which is an arcwise decomposable continuum, then $NC^{*}(X)=C(X)$. Therefore, the hypothesis of local connectedness in (1) and (3) is necessary in Corollary \ref{two}.
We do not know if (2) holds without the hypothesis of local connectedness.
\begin{question}\label{three} 
Let $X$ be a continuum such that $NC^{*}(X)$ is homeomorphic to $X$. Is $X$ an arc?
\end{question}

\section{Connectedness of $NC^{*}(G)$, for a finite graph $G$}

A {\it simple triod} is a continuum homeomorphic to the cone over the discrete space containing exactly three points. The point that correspond to the top of the cone in a simple triod is called its vertex.  Given a finite graph $G$, a point $p\in G$ is a {\it ramification point} if is the vertex of a simple triod contained in $G$, and $p$ is an {\it end-point} of $G$ if there exists an arc $\alpha$ that is a neighborhood of $p$ in $G$ and $p$ is an end-point of $\alpha$. A {\it cycle} of $G$ is a simple closed
curve $S$ in $G$ such that there exists a point $p\in S$ satisfying that
$S\setminus \{p\}$ is open in $G$. A point $p\in G$ is a {\it vertex} of $G$ if $p$ is either a ramification point or an end-point of $G$. The set of vertices of $G$ is denoted by $V(G)$. An $\it edge$ of $G$ is the closure of a component of $G\setminus V(G)$. The set of edges of $G$ is denoted by $E(G)$. Observe that each edge of $G$ is either an arc or a simple closed curve. A {\it cut-point} of $G$ is a point $p\in G$ such that $G\setminus \{p\}$ is not connected (equivalently, $\{p\}\notin NC^{*}(G)$). A {\it tangle} is a finite graph without cut-points. Observe that if $G$ is a tangle and $G$ is not a simple closed curve, then $G$ does not contain end-points nor cycles. Given two points $p$ and $q$ in the finite graph $G$, the notation $pq$, will denote some arc in $G$ with end-points $p$ and $q$.

Consider the family $\mathcal{FT}$ of finite graphs $G$ that are of one of the following forms:
\begin{itemize}
\item{(A)} $G$ is an arc,
\item{(B)} $G$ is a tangle,
\item{(C)} $G$ is a tangle with a sticker (that is, $G=K\cup L$, where $K$ is a tangle, $J$ is an arc and $K\cap J=\{p\}$, where $p$ is an end-point of $J$),
\item{(D)} $G$ is the junction of two tangles (that is, $G=K\cup M$, where $K$ and $M$ are tangles and $K\cap M$ is a one-point set),
\item {(E)} $G$ is the union of two tangles joined by an arc (that is, $G=K\cup pq\cup M$, where $p$ and $q$ are points of $G$,\ $K$ and $M$ are tangles, $K\cap M=\emptyset$, $K\cap pq=\{p\}$ and $M\cap pq=\{q\}$).
\end{itemize}

The aim of this section is to prove that for a finite graph $G$, $NC^{*}(G)$ is connected if and only if $G$ belongs to $\mathcal{FT}$.

\begin{theorem} \label{four}
Let $G$ be a tangle and $p\in G$. Let $A\in NC^{*}(G)$ be such that $p\notin A$. Then there exists an path $\alpha$ in $NC^{*}(G)$ that joins $A$ to $G$ such that for each $B\in\operatorname{Im}(\alpha) \setminus \{G\}$, $p\notin B$.
\end{theorem}

\begin{proof}
Since Theorem \ref{four} is immediate when $G$ is a simple closed curve, from now on we suppose that $G $ is not a simple closed curve. So, $G$ does not contain cycles and each edge of $G$ is an arc. 

Let $U=G\setminus A$. 

{\bf Case 1.} $U\cap V(G)=\emptyset$. 

Since $U$ is connected, we have that there exists an edge $L$ of $G$ such that $U\subset L$ and there exist points $a\neq b$ in $L$ such that $ab=\operatorname{cl}_{G}(U)$. Since $G$ is a tangle, $G$ does not have end-points. So, $p\notin\{a,b\}$. Then there exist one-to-one mappings $f,g:[0,1]\rightarrow L$  such that $f(0)=a$, $f(1)=p$, $g(0)=b$ and $g(1)=p$. Define $\alpha:[0,1]\rightarrow NC^{*}(G)$ by $\alpha(t)=A\cup af(t)\cup bg(t)$. Clearly, $\alpha$ has the required properties. 

From now on, by Case 1, we may assume that $U\cap V(G)\neq \emptyset$.

Let
\begin{center} 
$S=\{w\in V(G)\cap U:$ there exists $L \in E(G)$ such that $w\in L$ and $L\cap
A\neq\emptyset\}$.
\end{center}

For each $w\in S$, let 

\begin{center}
$D(w)=\bigcup\{L\in E(G): w\in L$ and $L\cap A\neq \emptyset\}$.
\end{center}

We claim that $S\neq\emptyset$. Fix an element $w\in U\cap V(G)$ and an element
$a\in A$. Choose an arc $wa$ in $G$. Let $w_{1}$ be the last element in
the arc $wa$, walking from $w$ to $a$, that belongs to $v(G)\cap U$. Let $a_{1}$ be the first element of $w_{1}a$, walking from $w_{1}$ to $a$, that belongs to $A$. Observe that $w_{1}a_{1}\setminus \{w_{1},a_{1}\}$ does not contain vertices of $G$, so it is contained in an
edge $L$ of $G$ such that $w_{1}\in L$ and $L\cap A\neq\emptyset$. Therefore, $w_{1}\in S$ and $S\neq\emptyset$.

{\bf Case 2.} $S$ is a one-point set.

Let $v$ be the unique point in $S$. We claim that $V(G)\subset A\cup \{v\}$.  Suppose to the contrary that there exists a point $w\in V(G)\setminus (A\cup \{v\})$, since $G\setminus \{v\}$ is arcwise connected, we can choose an arc $wa$ in $G\setminus \{v\}$. Proceeding as in the proof that $S\neq\emptyset$, we conclude that there exists an element $w_{1}\in S\cap wa$. Then $w_{1}\in S\setminus \{v\}$, a contradiction. We have shown that $V(G)\subset A\cup \{v\}$. 

Now, we claim that $G=A\cup D(v)$.
 Take point $q\in G\setminus(A\cup \{v\})$. Let $b,c\in V(G)$ be such that $q$ belongs to the edge $bc$. By the previous paragraph, $b,c\in A\cup \{v\}$. In the case that $b,c\in A$, since $U=G\setminus A$ is connected and $U\cap bc\neq\emptyset$, we obtain that $U\subset bc\setminus\{b,c\}$. So, $U\cap V(G)=\emptyset$, a contradiction. Hence, we may assume that $b\in A$ and $c\notin A$. Then $c\in S$, so $c=v$ and $q\in D(v)$. We have shown that $G=A\cup D(v)$. 

Given $L\in E(G)$ such that $v\in L$ and $L\cap A\neq \emptyset$, suppose that $L=vz$ for some $z\in V(G)$. If $z\notin A$, we have that $z\neq v$ and $z\in S$. Thus, $S\neq \{v\}$, a contradiction. We have proved that $z\in A$.

By the previous paragraph, there exist $m\in \mathbb{N}$ and $a_{1},\ldots,a_{m}\in A$ such that $\{L\in E(G):v\in L$ and $L\cap A\neq \emptyset\}=\{va_{1},\ldots,va_{m}\}$. If $m=1$, then $G=A\cup D(v)=A\cup va_{1}$. Let $a$ be the first point of $va_{1}$, walking from $v$ to $a_{1}$ such that $a\in A$. Then $G=va\cup aa_{1}\cup A$ and $va\cap
(aa_{1}\cup A)=\{a\}$. Moreover, since $v\notin A$ and $G$ is not an arc, we have that $a$ is a cut-point of $G$, a contradiction. Hence, $m\geq 2$. 

For each $i\in \{1,\ldots,m\}$ fix a one-to-one mapping $\alpha_{i}:[0,1]\rightarrow a_{i}v$ such that $\alpha_{i}(0)=a_{i}$ and $\alpha_{i}(1)=v$. 

{\bf Subcase 2.1.} $p\neq v$. 

Since $p\notin A$, we may assume that $p\in va_{1}\setminus \{v,a_{1}\}$.

Let $\sigma:[0,1]\rightarrow C(G)$ be defined as
\begin{center}
$\sigma(t)=A\cup a_{2}\alpha_{2}(t)\cup\cdots \cup a_{m}\alpha_{m}(t)$.
\end{center}

 It is easy to show that $\sigma$ is continuous, for each $t\in [0,1]$, $\sigma(t)\in NC^{*}(G)$ and $p\notin \sigma (t)$, and $U_{1}=G\setminus \sigma(1)$ is an open connected subset of $G$ such that $U_{1}\cap V(G)=\emptyset$ and $p\in U_{1}$. Thus, we can proceed as in Case 1 to obtain a path $\alpha_{0}$ in $NC^{*}(G)$ that joins $\sigma(1)$ to $G$
such that for each $B\in\operatorname{Im}(\alpha_{0}) \setminus \{G\}$, $p\notin B$.
Therefore, using the paths $\sigma$ and $\alpha_{0}$, it is possible to construct the required path $\alpha$.

{\bf Subcase 2.2} $v=p$.     

In this case, let $\alpha:[0,1]\rightarrow C(G)$ be defined as
\begin{center}
$\alpha(t)=A\cup a_{1}\alpha_{1}(t)\cup\cdots \cup a_{m}\alpha_{m}(t)$.
\end{center}
 
It is easy to check that, $\alpha$ has the required properties.
 
{\bf Case 3.} $S$ has more than one point.

For this case, we will need the following claim.

{\bf Claim 1.} If $S$ has $m$ points ($m\geq 2$), then there exist $w_{0}\in S$ and a mapping $\gamma:[0,1]\rightarrow NC^{*}(G)$ such that $\gamma(0)=A$, $A\cup\{w_{0}\}\subset \gamma(1)$ and for each $t\in [0,1]$, $p\notin \gamma(t)$.

Given $w\in S$, for each $L\in E(G)$ such that $w\in L$ and $L\cap A\neq \emptyset$, let $a_{L}$ be the first element in the arc $L$, going from $w$ to  the set $A$ that belongs to $A$. Let
\begin{center}
$D^{*}(w)=\bigcup\{wa_{L}\subset L:L\in E(G)$, $w\in L$ and $L\cap A\neq \emptyset\}$.
\end{center}

Observe that $D^{*}(w)\setminus A=\bigcup\{wa_{L}\setminus\{a_{L}\}\subset L:L\in E(G)$, $w\in L$ and $L\cap A\neq
\emptyset\}$, so $D^{*}(w)$ is connected.

Given $v,w\in S$ with $v\neq w$. we claim that $(D^{*}(v)\setminus A)\cap (D^{*}(w)\setminus A)=\emptyset$. Suppose to the contrary that there exist $L,J\in E(G)$ such that $w\in L$, $v\in J$, $L\cap A\neq \emptyset \neq J\cap A$ and a point $q\in (wa_{L}\setminus\{a_{L}\})\cup (va_{J}\setminus\{a_{J}\})$. Since $q\notin \{a_{L},a_{J}\}$, it follows that either $w=q=v$ or $L=J$ and $q\in wa_{L}\setminus\{a_{L}\}$. In both cases $w=v$, a contradiction.

Let $S^{*}=\{w\in S:p\notin D^{*}(w)\}$. Since $p\notin A$, the conclusion of the previous paragraph implies that $S^{*}\neq\emptyset$.

{\bf Subclaim 1.1.} There exists $w_{0}\in S^{*}$ such that $U\setminus D^{*}(w_{0})$ is connected. 

Suppose contrary to Subclaim 1.1, that for each $w\in S^{*}$, there exist two nonempty disjoint open subsets $Y_{w}$ and $Z_{w}$ of $G$ such that $U\setminus D^{*}(w)=Y_{w}\cup Z_{w}$. Since $p\notin D^{*}(w)$, we assume that $p\in Z_{w}$. So, $p\notin Y_{w}$.

Choose and fix an element $w_{1}\in S^{*}$. 

{\bf Subclaim 1.1.1.} $Y_{w_{1}}\cap S^{*}\neq \emptyset$.

In order to prove this subclaim, take an element $q\in Y_{w_{1}}$. Since $w_{1}$ is not a cut-point of $G$, there exists an arc $qa$ in $G\setminus\{w_{1}\}$, where $a\in A$. We may assume that $qa\cap A=\{a\}$. Then $qa\setminus\{a\}\subset U$. 

If there exists a point $x\in(qa\setminus\{a\})\cap D^{*}(w_{1})$, then there exists $L\in E(G)$ such that $w_{1}\in L$, $L\cap A\neq\emptyset$ and $x\in w_{1}a_{L}\setminus\{a_{L}\}$. Since $x\neq w_{1}$, we have that $x$ is a point in the edge $L$ and $x$ is not an end-point of $L$. Since the only way the arc $qx$ can enter to the arc $w_{1}a_{L}$ is by $w_{1}$ or by $a_{L}$, we have that either the arc $qx$ is contained in $w_{1}a_{L}$ or $w_{1}\in qx$ or $a_{L}\in qx$. This is impossible since $q\notin D^{*}(w_{1})$, $w_{1}\notin qa$ and $qx\cap A = \emptyset$.
We have shown that $qa\setminus\{a\}$ is a connected subset of $U\setminus D^{*}(w_{1})$. Thus, $qa\setminus \{a\}\subset Y_{w_{1}}$. 

In the case that $(qa\setminus\{a\})\cap V(G)\neq\emptyset$, we can take the last element $w$ in $qa\setminus\{a\}$, going from $q$ to $a$, that belongs to $V(G)$. Then $w\in S\cap Y_{w_{1}}$. Since $D^{*}(w)\setminus A$ is a connected subset of $U\setminus D^{*}(w_{1})$ that intersects $Y_{w_{1}}$, we conclude that $p\notin D^{*}(w)$. Therefore, $w\in Y_{w_{1}}\cap S^{*}$ and we are done. 

Suppose then that $(qa\setminus\{a\})\cap V(G)=\emptyset$. Then there exists $J\in E(G)$ such that, if $J=yz$, then\ $qa\setminus\{a\}\subset yz\setminus \{y,z\}$ and $q\in ya$. If there exists a point $a_{1}\in(ya\setminus \{a\})\cap A$, then $a_{1}\notin qa$. Thus, $q\in a_{1}a\cap U\subset J\cap U$. The connectedness of $U$ implies that $U\subset a_{1}a\subset J$. Hence, $U\cap V(G)=\emptyset$, contrary to our assumption. This proves that $(ya\setminus \{a\})\cap
A=\emptyset$ and $ya\setminus \{a\}\subset U$. 

We check that $(ya\setminus\{a\})\cap D^{*}(w_{1})=\emptyset$. Suppose to the contrary that there exist a point $z\in  ya\setminus\{a\}$ and an edge $L\in E(G)$ such that  $w_{1}\in L$, $L\cap A\neq\emptyset$ and $z\in w_{1}a_{L}\setminus\{a_{L}\}$.
Observe that $y$ and $w_{1}$ are vertices of $G$, and $yz$ and $w_{1}z$ are proper (possibly degenerate) subarcs of edges of $G$. This implies that $y=w_{1}$. Then the edge $J$ satisfies $y=w_{1}\in J$ and $a\in J\cap A$. By the previous paragraph, $a=a_{J}$. Hence, $q\in ya=w_{1}a_{J}\subset D^{*}(w_{1})$, contradicting the fact that $q\in Y_{w_{1}}$. This completes the proof that $(ya\setminus\{a\})\cap D^{*}(w_{1})=\emptyset$.

We have obtained that $ya\setminus\{a\}$ is a connected subset of $U\setminus  D^{*}(w_{1})$ and intersects $Y_{w_{1}}$. Hence, $ya\setminus\{a\}\subset Y_{w_{1}}$. Therefore, $y\in S$ and $y\neq w_{1}$. This implies that $(D^{*}(y)\setminus A)\cap
(D^{*}(w_{1})\setminus A)=\emptyset$. Hence, $D^{*}(y)\setminus
A$ is a connected subset of $U\setminus D^{*}(w_{1})$ and intersects $Y_{w_{1}}$. We conclude that $D^{*}(y)\setminus A\subset Y_{w_{1}}$. Since $p\notin Y_{w_{1}}\cup A$, we conclude that $p\notin D^{*}(y)$. Therefore, $y\in Y_{w_{1}}\cap S^{*}$.
This completes the proof that $Y_{w_{1}}\cap S^{*}\neq \emptyset$.
Thus, Subclaim 1.1.1. is proved.

Choose and fix an element $w_{2}\in Y_{w_{1}}\cap S^{*}$. Then $w_{2}\neq w_{1}$, so $(D^{*}(w_{2})\setminus
A)\cap
(D^{*}(w_{1})\setminus A)=\emptyset$. Thus, $D^{*}(w_{2})\setminus
A$ is a connected subset of $U\setminus D^{*}(w_{1 })$ and intersects $Y_{w_{1}}$. Then $D^{*}(w_{2})\setminus A\subset Y_{w_{1}}$ and $(D^{*}(w_{1})\setminus A)\cup Z_{w_{1}}\subset U\setminus D^{*}(w_{2})=Y_{w_{2}}\cup Z_{w_{2}}$. Observe that $D^{*}(w_{1})\setminus A$ is a connected subset of the connected set  $U$ and $U\setminus (D^{*}(w_{1})\setminus A)=U\setminus D^{*}(w_{1})=Y_{w_{1}}\cup Z_{w_{1}}$. By \cite[Proposition  6.3, p.88]{n}, we obtain that $(D^{*}(w_{1})\setminus A)\cup Z_{w_{1}}$ is a connected subset of $Y_{w_{2}}\cup Z_{w_{2}}$ that contains the point $p\in Z_{w_{2}}$. Thus, $(D^{*}(w_{1})\setminus A)\cup Z_{w_{1}}\subset Z_{w_{2}}$ and $Y_{w_{2}}\subset U\setminus Z_{w_{2}}\subset U\setminus ((D^{*}(w_{1})\setminus A)\cup Z_{w_{1}})\subset Y_{w_{1}}$. Therefore, $Y_{w_{2}}\subset Y_{w_{1}}$.

Proceeding with $w_{2}$ as we did with $w_{1}$, we can prove that there exists an element $w_{3}\in Y_{w_{2}}\cap S^{*}$ and $Y_{w_{3}}\subset Y_{w_{2}}\subset Y_{w_{1}}$. Hence, $w_{3}\notin\{w_{1},w_{2}\}$. 

This procedure can be repeated to obtain a sequence $\{w_{k}\}_{k=1}^{\infty}$ in $S^{*}$ such that for each $k\in\mathbb{N}$, $w_{k+1}\in Y_{w_{k}}$ and $Y_{w_{k+1}}\subset Y_{w_{k}}$. Since for each $k\in\mathbb{N}$, $w_{k}\notin Y_{w_{k}}$, we conclude that the points $w_{k}$ are pairwise distinct.
This contradicts the fact that $S^{*}$ is finite and completes the proof of Subclaim 1.1. is proved.

According Subclaim 1.1, we can choose and fix $w_{0}\in S^{*}$ such that $U\setminus D^{*}
(w_{0})$ is connected.

Suppose that $L_{1},\ldots,L_{m}$ are the edges of $G$ such that for each $i\in\{1,\ldots,m\}$, $w_{0}\in L_{i}$ and $A\cap L_{i}\neq\emptyset$. Given $i\in \{1,\ldots,m\}$, let $\gamma_{i}:[0,1]\rightarrow a_{L_{i}}w_{0}$ be a one-to-one mapping such that $\gamma_{i}(0)=a_{L_{i}}$ and $\gamma_{i}(1)=w_{0}$.
Define $\gamma:[0,1]\rightarrow C(G)$ by:
\begin{center}
$\gamma(t)=A\cup a_{L_{1}}\gamma_{1}(t)\cup\cdots\cup a_{L_{m}}\gamma_{m}(t)$.
\end{center}

Clearly, $\gamma$ is a one-to-one continuous function,  $A\cup \{w_{0}\}\subset A\cup D^{*}(w_{0})=\gamma(1)$ and $\gamma(0)=A$. Since for each $t\in [0,1]$, $\gamma(t)\subset A\cup D^{*}(w_{0})$, we have that $p\notin\gamma(t)$. Then we only left to prove that for each $t\in [0,1]$, $G\setminus\gamma(t)$ is connected.
Let $t\in [0,1]$. 

Since $U=G\setminus A$ is connected and $p,w_{0}\in U$, there exists an arc $pw_{0}$ in $U$. Inside $pw_{0}$ we can take a point $q_{0}\in D^{*}(w_{0})\cap pw_{0}$ such that the subarc $pq_{0}$ satisfies that $D^{*}(w_{0})\cap pq_{0}=\{q_{0}\}$. Then there exists $L\in E(G)$ such that $w_{0}\in L$, $L\cap A\neq\emptyset$ and $q_{0}\in w_{0}a_{L}$. Since $p\notin D^{*}(w_{0})$ and the only way to enter to the arc $w_{0}a_{L}$ is by one of the points $w_{0}$ or $a_{L}$ and $q_{0}\notin A$, we conclude that $q_{0}=w_{0}$.

If $t=1$, then $G\setminus \gamma(t)=G\setminus(A\cup D^{*}(w_{0}))=U\setminus D^{*}(w_{0})$ is connected. So, suppose that $t<1$. In this case, 
\begin{center}
$G\setminus \gamma(t)=(U\setminus D^{*}(w_{0}))\cup (w_{0}\gamma_{1}(t)\setminus \{\gamma_{1}(t)\})\cup\cdots\cup (w_{0}\gamma_{m}(t)\setminus \{\gamma_{m}(t)\})$.   
\end{center}

Thus, $G\setminus\gamma(t)$ is the union of connected sets and each one intersects the connected set $pw_{0}\subset G\setminus \gamma(t)$. Therefore, $G\setminus \gamma(t)$ is connected and $\gamma(t)\in NC^{*}(G)$. This completes the proof of Claim 1.

\bigskip

Set $A_{1}=\gamma(1)$. We can define the respective sets $U_{1}$ ($U_{1}=G\setminus A_{1}$); $S_{1}\subset E(G)\setminus A_{1}$; and for each $w\in S_{1}$, the respective set $D_{1}(w)$. 

Proceeding with the set $A_{1}$ as we did with the set $A$, we can find the required path $\alpha$ for all cases, except for the case that $S_{1}$ has more than one point. So, we can apply again Claim 1 (to $A_{1}$, $U_{1}$ and $S_{1}$) to obtain a corresponding point $w_{0}^{(1)}$ and a corresponding mapping $\gamma^{(1)}$. Then we can define $A_{2}=\gamma^{(1)}(1)$, and continue proceeding in a similar way. This procedure finishes since the points we are obtaining $w_{0},w_{0}^{(1)},w_{0}^{(2)},\ldots$ are pairwise distinct and they are elements of $V(G)$. This ends the proof of Theorem \ref{four}.     
\end{proof}

\begin{lemma}\label{five}
Let $X$ be a continuum. Suppose that there exists a point $p\in X$ such
that $X\setminus \{p\}$ has at least three components. Then $NC^{*}(X)$ is
not connected. 
\end{lemma} 

\begin{proof}
By hypothesis, there exist nonempty pairwise disjoint open subsets $U_{1}$, $U_{2}$ and $U_{3}$ of $X$ such that $X\setminus \{p\} =U_{1}\cup U_{2}\cup U_{3}$. By \cite[Theorem 6.6, p.89]{n}, for each $i\in\{1,2,3\}$, there exists a point $p_{i}\in U_{i}$ such that $\{p_{i}\}\in NC^{*}(X)$.

Let $\mathcal{W}=\{A\in NC^{*}(X):A\subset U_{1}\}$. Then $\mathcal{W}$ is an open nonempty proper subset of $NC^{*}(X)$. We claim that $\mathcal{W}$ is closed in $NC^{*}(X)$. Suppose to the contrary that there exists $B\in\operatorname{cl}_{NC^{*}(X)}(\mathcal{W})\setminus \mathcal{W}\subset \operatorname{cl}_{C(X)}(\mathcal{W})$. Then $B\subset \operatorname{cl}_{X}(U_{1})=U_{1}\cup\{p\}$ and $B\nsubseteq U_{1}$. Thus, $X\setminus B=U_{2}\cup U_{3}\cup (U_{1}\setminus B)$ is not connected. This contradicts the fact that $B\in NC^{*}(X)$. We have shown that $\mathcal{W}$
is closed in $NC^{*}(X)$. Therefore, $NC^{*}(X)$ is not connected.  
\end{proof}

\begin{theorem}\label{six}
Let $G$ be a finite graph. Then the following are equivalent:
\begin{itemize}
\item{(a)} $G\in\mathcal{FT}$,
\item{(b)} $NC^{*}(G)$ is arcwise connected, and 
\item{(c)} $NC^{*}(G)$ is connected.
\end{itemize}
\end{theorem}

\begin{proof}
$(a) \Rightarrow (b)$. Take $G\in \mathcal{FT}$. We only consider the case that{} $G$ is the union of two tangles joined by an arc. The other cases can be treated in a similar way. Suppose then that $G=K\cup
pq\cup M$, where $p$ and $q$ are points of $G$,\ $K$ and $M$ are tangles,
$K\cap M=\emptyset$, $K\cap pq=\{p\}$ and $M\cap pq=\{q\}$. 
 
Take an element $A\in NC^{*}(G)$. We are going to show that there exists a path $\sigma:[0,1]\rightarrow NC^{*}(G)$ such that $\sigma(0)=A$ and $\sigma(1)= G$. We consider three cases.

{\bf Case 1.} $A\subset K\setminus\{p\}$.

Since $K$ is a tangle, by Theorem \ref{four}, there exists a mapping $\sigma_{1}:[0,1]\rightarrow NC^{*}(K)$ such that $\sigma_{1}(0)=A$, $\sigma_{1}(1)=K$ and for each $t\in [0,1)$, $p\notin \sigma_{1}(t)$.

Given $t\in[0,1)$, $G\setminus \sigma_{1}(t)=M\cup L\cup (K\setminus \sigma_{1}(t))$ is connected. Moreover, $G\setminus K=M\cup (L\setminus \{p\})$ is also connected. Therefore, for each $t\in [0,1]$, $\sigma_{1}(t)\in NC^{*}(G)$.

Let $\gamma:[0,1]\rightarrow L$ be a one-to-one mapping such that $\gamma(0)=p$ and $\gamma(1)=q$. Let $\sigma_{2}:[0,1]\rightarrow C(G)$ be given by $\sigma_{2}(t)=K\cup p\gamma(t)$. Then $\sigma_{2}$ is continuous, $\sigma_{2}(0)=K$ and $\sigma_{2}(1)=K\cup L$ . Given $t\in [0,1)$, $G\setminus \sigma_{2}(t)=M\cup(\gamma(t)q\setminus\{\gamma(t)\})$ is connected. Moreover, $G\setminus \sigma_{2}(1)=M\setminus\{q\}$ is connected, since $M$ is a tangle. Therefore, for each $t\in [0,1]$, $\sigma_{2}(t)\in NC^{*}(G)$.

Apply again Theorem \ref{four}, to obtain a mapping $\lambda:[0,1]\rightarrow NC^{*}(M)$ such that $\lambda(0)=\{p\}$ and $\lambda(1)=M$. Let $\sigma_{3}:[0,1]\rightarrow C(G)$ be given by $\sigma_{3}(t)=K\cup L\cup\lambda(t)$. Then $\sigma_{3}(0)=K\cup L$, $\sigma_{3}(1)=G$ and for each $t\in [0,1]$, $G\setminus \sigma_{3}(t)=M\setminus \lambda(t)$ is connected. Therefore, $\sigma_{3}(t)\in NC^{*}(G)$.

Finally, combining the mappings $\sigma_{1}$, $\sigma_{2}$ and $\sigma_{2}$, the required mapping $\sigma$ can be obtained.

{\bf Case 2.} $A\subset M\setminus\{q\}$.

This case can be treated in a similar way as Case 1.

{\bf Case 3.} $A\cap L\neq \emptyset$.

Fix a point $a\in A\cap L$. Then $G\setminus A$ is a connected subset of the union of the separated sets $(K\cup pa)\setminus\{a\}$ and $(M\cup qa)\setminus\{a\}$. This implies that either $K\cup pa\subset A$ or $M\cup qa\subset A$. We may suppose that $K\cup pa\subset A$. Then $A$ is of one of the forms: (i) $A=K\cup pe$ for some $e\in pq\setminus\{q\}$, or (ii) $K\cup L\cup E$, where $E\in C(M)$ and $q\in E$. In the subcase (i), we can use mappings similar to $\sigma_{2}$ and $\sigma_{3}$ in Case 1 to find the required $\sigma$. In the subcase (ii), we have that $M\setminus E=G\setminus A$ is connected. Applying Theorem 4, we obtain a mapping $\xi:[0,1]\rightarrow NC^{*}(M)$
such that $\xi(0)=E$ and $\xi(1)=M$. Thus, the mapping $\sigma:[0,1]\rightarrow
C(G)$ given by $\sigma(t)=K\cup L\cup\xi(t)$ satisfies the required conditions. 

$(c) \Rightarrow (a)$. By Lemma \ref{five}, for each $p\in G$, $G\setminus \{p\}$ has exactly one or two components. Since tangles are in $\mathcal{FT}$, we suppose that $G$ contains cut-points.

{\bf Case 1.} $G$ contains exactly one cut-point.

Let $p$ be the only cut-point of $G$. Then $G\setminus \{p\}$ has exactly two components. Then there exist two nonempty disjoint open subsets $U$ and $V$ of $G$ such that $G\setminus \{p\}=U\cup V$. By Lemma \ref{five}, $U$ and $V$ are connected. By \cite[Proposition  6.3, p.88]{n}, the sets $K=U\cup\{p\}$ and $M=V\cup\{p\}$ are subcontinua of $G$. If $K$ is not a tangle, there exists a point $q\in U$ such that $K\setminus \{q\}=Y\cup Z$, for some nonempty disjoint open subsets of $K$. We may assume that $p\in Z$. Then $Y$ and $Z\cup V$ are separated subsets of $G$ such that $G\setminus \{q\}=Y\cup (Z\cup V)$. Hence, $q$ is also a cut-point of $G$. This contradicts our assumption for this case and proves that $K$ is a tangle. Similarly, $M$ is a tangle. Thus, $G$ is the junction of two tangles and $G\in\mathcal{FT}$.

{\bf Case 2.} $G$ contains at least two cut-points.

Let $p\neq q$ be two cut-points of $G$. In this case, $G$ is not a simple closed curve. 

Then there exist two nonempty disjoint open subsets
$U$ and $V$ of $G$ such that $G\setminus \{p\}=U\cup V$.
We assume that $q\in V$. Let $Y$ and $Z$ be nonempty disjoint open subsets
of $G$ such that $G\setminus \{q\}=Y\cup Z$. By Lemma \ref{five}, $U$, $V$, $Y$ and $Z$ are connected. We assume that $p\in Y$.

By \cite[Proposition  6.3, p.88]{n}, $U\cup\{p\}$, $Y\cup\{q\}$ and $Z\cup\{q\}$ are subcontinua of $G$. Observe that $U\cup\{p\}\subset G\setminus \{q\}=Y\cup Z$. Since $p\in Y$, we have that $U\cup\{p\}\subset Y$. Similarly, $Z\cup\{q\}\subset V$.

Observe that $Y\cup\{q\}=U\cup\{p\}\cup((V\cap Y)\cup\{q\})$ and $(Y\cup\{q\})\setminus\{p\}=U\cup((V\cap Y)\cup \{q\})$. Since $U$ and $(V\cap Y)\cup \{q\}$ are separated, by \cite[Proposition  6.3, p.88]{n}, the set

\begin{center}
$T=\{p\}\cup (V\cap Y)\cup\{q\}$
\end{center}
 is a continuum.

On the other hand, $Y=U\cup\{p\}\cup(V\cap
Y)$ and $Y\setminus\{p\}=U\cup(V\cap Y)$.
Since $U$ and $V\cap Y$ are separated, by \cite[Proposition  6.3, p.88]{n}, $\{p\}\cup
(V\cap Y)$ is connected. Thus, $q$ is not a cut-point of $T$. Moreover, since $V=(V\cap Y)\cup\{q\}\cup(V\cap Z)$, we have that $(V\cap Y)\cup\{q\}$ is connected. Thus, $p$ is not a cut-point of $T$.

We check that $T$ is an arc with end-points $p$ and $q$. Suppose the contrary. By \cite[Theorem 6.17, p.96]{n}, $T$ contains a non-cut-point $x\notin\{p,q\}$. Since $G\setminus\{x\}=(T\setminus\{x\})\cup(U\cup\{p\})\cup (Z\cup \{q\})$, we conclude that $x$ is not a cut-point of $G$. So, $\{x\}\in NC^{*}(G)$. Let $\mathcal{W}=\{A\in NC^{*}(G):A\subset V\cap Y\}$. Then $\mathcal{W}$ is open and nonempty in $NC^{*}(G)$. We claim that $\mathcal{W}$ is also closed in $NC^{*}(G)$. Suppose to the contrary that there exists an element $A\in\operatorname{cl}_{NC^{*}(G)}(\mathcal{W})\setminus\{\mathcal{W}\}\subset \operatorname{cl}_{C(G)}(\mathcal{W})$. Then $A\subset\operatorname{cl}_{G}(V\cap Y)=T$. This implies that $U\cup Z\subset G\setminus A$. Since $A\nsubseteq V\cap Y$, we may assume that $p\in A$. Then $G\setminus A\subset U\cup V$, $U\subset G\setminus A$ and $Z\subset V\cap (G\setminus A)$. Hence, $G\setminus A$ is not connected and $A\notin NC^{*}(G)$, a contradiction. Therefore, $\mathcal{W}$ is a nonempty proper open and closed subset of $NC^{*}(G)$, contradicting our assumption. This ends the proof that $T$ is an arc with end-points $p$ and $q$.

Since $T\setminus \{p,q\}=V\cap Y$, we have that $T\setminus\{p,q\}$ is open in $G$. Thus, $T\setminus\{p,q\}$ is arcwise connected and contains no vertices of $G$. This implies that there exists an edge $L$ of $G$ such that $pq\subset L$. If $L$ is a cycle, since $G$ is not a simple closed curve, we have that there exists a point $x_{0}\in L$, such that $\operatorname{Fr}_{G}(L)=\{x_{0}\}$ and $x_{0}$ is the only cut-point of $G$ belonging to $L$. This contradicts that $\{p,q\}\subset L$. Thus, $L$ is an arc. Let $p_{0}$ and $q_{0}$ be the end-points of $L$. We suppose that we have an order $<$ in $L$ such that $p_{0}<q_{0}$. We may suppose also that $p<q$. Then $q\notin p_{0}p$ and $p\notin qq_{0}$.

Let $L'=L\setminus\{p_{0},q_{0}\}$. Let $D$ (resp., $E$) be the component of $G\setminus L'$ containing $p_{0}$ (resp., $q_{0}$). Then $D$ and $E$ are subcontinua of $G$.

Observe that
\begin{center}
$G=U\cup\{p\}\cup(V\cap Y)\cup\{q\}\cup Z=U\cup pq\cup Z$.
\end{center}

If $p_{0}\in\{q\}\cup Z$, since $p\in Y$, we have that $q\in p_{0}p$, a contradiction. This implies that $p_{0}\in U\cup\{p\}$. Similarly, $q_{0}\in \{q\}\cup Z$. Since $D\subset G\setminus (p_{0}q_{0}\setminus\{p_{0},q_{0}\})\subset G\setminus (pq\setminus\{p,q\})=(U\cup\{p\})\cup(\{q\}\cup Z)$, the sets $U\cup\{p\}$ and $\{q\}\cup Z$ are closed and disjoint; and $p_{0}\in D\cap (U\cup\{p\})$, we have that $D\subset U\cup\{p\}$. Similarly, $E\subset \{q\}\cup Z$. Therefore, $D\cap E=\emptyset$. 

Given $x\in G\setminus L$, let $xy$ be an arc in $G$ such that $y\in L$ and $xy\cap L=\{y\}$. Then $y\in\{p_{0},q_{0}\}$. Since $xy\subset G\setminus L'$, we conclude that $x\in D\cup E$. Therefore, 

\begin{center}
$G=D\cup p_{0}q_{0}\cup E$.
\end{center}
 Observe that $D\cap p_{0}q_{0}=\{p_{0}\}$, $E\cap p_{0}q_{0}=\{q_{0}\}$ and the only arc joining $p$ and $q$ is the subarc of $p_{0}q_{0}$ joining them.

If we take a cut-point $z$ of $G$ such that $z\notin\{p,q\}$, we can proceed with the pairs $p,z$ and $q,z$ in the same way as we did with the pair and obtain that there exist only one arc $pz$ and only one arc $pz$, and the sets $pz\setminus\{p,z\}$ and $qz\setminus\{q,z\}$ contain no vertices of $G$. This implies that $z\in p_{0}q_{0}$. We have shown that the set of cut-points of $G$ is contained in $p_{0}q_{0}$. 

We claim that if $D$ is non-degenerate, then $D$ is a tangle. Suppose to the contrary that there exists a point $z\in D$ such that $D\setminus\{z\}=R\cup S$, where $R\neq\emptyset\neq S$, and $R$ and $S$ are separated in $D$ (and then they are separated in $G$). In the case that $z=p_{0}$, we have that $G\setminus\{p_{0}\}=R\cup S\cup ((p_{0}q_{0}\setminus\{p_{0}\})\cup E)$. Since the sets $R$, $S$ and $(p_{0}q_{0}\setminus\{p_{0}\})\cup E$ are nonempty and they are pairwise separated, we infer that $G\setminus \{p_{0}\}$ has at least three components. This contradicts Lemma \ref{five} and shows that $z\neq p_{0}$. Then we may suppose that $p_{0}\in S$. Thus, $G\setminus\{z\}=R\cup(S\cup p_{0}q_{0}\cup E)$ and the sets $R$ and $S\cup p_{0}q_{0}\cup E$ are nonempty and separated. Hence, $z$ is a cut-point of $G$. By the previous paragraph, $z\in p_{0}q_{0}$, a contradiction. Therefore, $D$ is a tangle.

Similarly, if $E$ is non-degenerate, then $E$ is a tangle.

We are ready to finish the proof of the theorem. If $D$ and $E$ are degenerate, then $G$ is an arc; if only one of the sets $D$ and $E$ is degenerate, then $G$ is a tangle with a sticker; and if both sets $D$ and $E$ are non-degenerate, then $G$ is the union of two tangles joined by an arc. In any case $G\in \mathcal{FT}$. \end{proof}

\section{Two more questions}

A {\it dendroid} is an arcwise connected continuum $X$ such that for every $A,B\in C(X)$, $A\cap B$ is connected. A dendroid containing exactly one ramification point is a {\it fan}.

\begin{question}\label{seven}
Let $X$ be a dendroid such that $NC^{*}(X)$ is compact. Is $X$ an arc?
\end{question} 

\begin{question}\label{eight}
Is there a fan $X$ such that $NC^{*}(X)$ is compact?
\end{question}

\bigskip

{\bf Acknowledgments} This paper was partially supported by the projects "Sistemas Din\'{a}micos discretos y Teor\'{i}a de Continuos" PAPIIT-IN105624 from program DGAPA of Universidad Nacional Aut\'onoma de M\'exico (UNAM)
and ``Teor\'{\i}a de Continuos e Hiperespacios, dos" (AI-S-15492) of CONAHCYT. Also, research of the third-named author was supported by "Programa de Estancias Posdoctorales por M\'exico 2023 (1)" (CVU 661607) of CONAHCYT and wishes to thank his colleagues Eduardo Garc\'{i}a-Mu\~noz and Ra\'ul Escobedo at Benem\'erita Universidad Aut\'onoma de Puebla (BUAP).

\bigskip

Addresses: 

Alejandro Illanes and Ver\'onica Mart\'inez-de-la-Vega: Instituto de Matem\'aticas, Universidad Nacional Aut\'onoma de M\'exico, Circuito Exterior, Cd. Universitaria, M\'exico 04510, Cd, de M\'exico,
M\'exico.

Jorge E. Vega: Benem\'erita Universidad Aut\'onoma de Puebla, Facultad de Ciencias F\'{i}sico-Matem\'aticas, Av. San Claudio y 18 sur, Col. San Manuel, Ciudad Universitaria, C.P. 72570, Puebla, Puebla, México.

email addresses: 

A. Illanes: illanes@matem.unam.mx

V. Mart\'inez-de-la-Vega: vmvm@matem.unam.mx

J.E. Vega: vegacevedofcfm@fcfm.buap.mx


\begin{thebibliography}{10}


\bibitem{ees} Ra\'ul Escobedo, Carolina Estrada-Obreg\'on; Javier S\'anchez-Mart\'inez,
{\it On hyperspaces of non-cut sets of continua}, Topology Appl. 217 (2017),
97-106. 


\bibitem{hmmv} Rodrigo Hern\'andez-Guti\'errez; Ver\'onica Mart\'{i}nez-de-la-Vega; Jorge M. Mart\'{i}nez-Montejano; Jorge E. Vega, {\it The hyperspace of noncut subcontinua of graphs and dendrites}, Colloq. Math. 173 (2023), no. 1, 57--75.

\bibitem{d} James Dugundji, Topology. Allyn and Bacon, Inc., Boston, MA, 1966.  

\bibitem{in} Alejandro Illanes; Sam B. Nadler, Jr., Hyperspaces, Fundamentals
and recent advances, Monographs and Textbooks in Pure and Applied Math. Vol.
216, Marcel Dekker, Inc. New York and Basel, 1999.

\bibitem{n} Sam B. Nadler, Jr., Continuum Theory, An introduction, Monographs
and Textbooks in Pure and Applied Math. Vol. 158, Marcel Dekker, Inc. New
York, N.Y., 1992.
\end{thebibliography}
\end{document}